\documentclass{amsart}
\usepackage{verbatim}
\usepackage{rotating} 
\usepackage{amssymb}
\usepackage{mathrsfs,amsfonts,amsmath,amssymb,epsfig,amscd,xy,amsthm,pb-diagram} 

\newtheorem{theorem}{Theorem}[section]
\newtheorem{proposition}[theorem]{Proposition}

\newtheorem{lemma}[theorem]{Lemma}

\newtheorem{corollary}[theorem]{Corollary}

\newtheorem{definition}[theorem]{Definition}
\newtheorem{problem}[theorem]{Problem}
\newtheorem{conjecture}[theorem]{Conjecture}

\theoremstyle{plain}

\theoremstyle{remark}

\newtheorem{remark}[theorem]{Remark}
\newtheorem{example}[theorem]{Example}

\newcommand{\C}{{\mathbb C}}

\newcommand{\Q}{{\mathbb Q}}

\newcommand{\R}{{\mathbb R}}
\newcommand{\Z}{{\mathbb Z}}
\newcommand{\N}{{\mathbb N}}

\newcommand{\cA}{{\mathcal A}}

\newcommand{\cK}{{\mathcal K}}
\newcommand{\cJ}{{\mathcal J}}

\newcommand{\Qbar}{\bar{\Q}}

\DeclareMathOperator{\Gal}{Gal}

\DeclareMathOperator{\Norm}{N}

\newcommand{\bP}{{\mathbb P}}

\newcommand{\cO}{\mathcal{O}}
\newcommand{\cD}{\mathcal{D}}

\newcommand{\cG}{\mathcal{G}}

\author{Khoa D.~Nguyen}
\address{
Khoa D.~Nguyen \\
Department of Mathematics and Statistics\\
University of Calgary\\
AB T2N 1N4, Canada
}
\email{dangkhoa.nguyen@ucalgary.ca}


\begin{document}
	\title{Transcendence of polynomial canonical heights}
	
	\date{29 December 2021}
	
	\begin{abstract}
	There are two fundamental problems motivated by Silverman's conversations over the years concerning the nature of the exact values of canonical heights of $f(z)\in\bar{\mathbb{Q}}(z)$ with $d:=\deg(f)\geq 2$. The first problem is the  
	conjecture that $\hat{h}_f(a)$ is either $0$ or transcendental for every $a\in \mathbb{P}^1(\bar{\mathbb{Q}})$; this holds when $f$ is linearly conjugate to
	$z^d$ or $\pm C_d(z)$ where $C_d(z)$ is the Chebyshev polynomial of degree $d$
	since $\hat{H}_f(a)$ is algebraic for every $a$. Other than this, very little is known: for example, it is not known if there \emph{exists} 
	even \emph{one} rational number $a$ such that  
	$\hat{h}_f(a)$ is \emph{irrational} where $f(z)=z^2+\displaystyle\frac{1}{2}$. The second problem asks for the characterization of all pairs $(f,a)$ such that $\hat{H}_f(a)$ is algebraic. In this paper,  
	we solve the second problem and obtain significant progress to the first problem in the case of polynomial dynamics. These are consequences of our main result concerning the possible algebraic numbers that can be expressed as a multiplicative combination of values of B\"ottcher coordinates. The proof of our main result uses a construction of a certain auxiliary polynomial and the powerful Medvedev-Scanlon classification of preperiodic subvarieties of split polynomial maps.
	\end{abstract}
	
	\maketitle
	
\section{Introduction}	
	For many decades, the theory of Weil heights and canonical heights
	has become not only an indispensable tool in diophantine geometry and arithmetic dynamics
	but also a highly interesting subject of its own. Yet the nature of the exact values of canonical height functions remains very mysterious. For the dynamics of univariate rational functions, we have the following:
	\begin{conjecture}[Silverman]\label{conj:0 or transcendental}
	Let $f(z)\in\Qbar(z)$ with $d:=\deg(f)\geq 2$. For every $a\in\bP^1(\Qbar)$, we have that $\hat{h}_f(a)$ is either $0$ or transcendental.
	\end{conjecture}
	
	As explained in Silverman's comments \cite{Sil13_MO}, this conjecture originates from his earlier
	conversations over the years about the canonical heights of non-torsion points on elliptic curves. 
	Conjecture~\ref{conj:0 or transcendental} together with some related problems and comments 
	also appear in \cite{Ngu15_AI}. Conjecture~\ref{conj:0 or transcendental} is 
	in stark contrast to known results over function fields by Chatzidakis-Hrushovski 
	\cite[Lemma~4.21]{CH08_DF1} in which values of canonical heights are usually algebraic (it appears 
	that in \cite{CH08_DF1}, the authors use the notation $H_{D}$ to denote a 
	certain \emph{logarithmic} canonical height) as well as more recent rationality 
	results by DeMarco-Ghioca \cite{DMG19_RO}.
	
	There is another problem that is an analogue of Conjecture~\ref{conj:0 or transcendental} for 
	the multiplicative canonical height $\hat{H}_f=\exp(\hat{h}_f)$. 
	A variant of the following has been suggested in Silverman's comments \cite{Sil13_MO} as well:
	\begin{problem}[Silverman]\label{prob:Silverman}
	Give a characterization of $(f,a)$ where $f(z)\in\Qbar(z)$ has degree $d\geq 2$ and $a\in\bP^1(\Qbar)$ such
	that $\hat{H}_f(a)$ is algebraic.
	\end{problem}
	
	Throughout this paper, $\N$ denotes the set of positive integers and $\N_0:=\N\cup\{0\}$. Let $C_d(z)$ be the Chebyshev polynomial of degree $d$: it satisfies the functional equation $\displaystyle C_d\left(z+\frac{1}{z}\right)=z^d+\frac{1}{z^d}$. When $f(z)\in\Qbar(z)$ is linearly conjugate to $z^d$ or $\pm C_d(z)$, we have that $\hat{H}_f(a)$ is algebraic for every $a\in \bP^1(\Qbar)$ and hence $\hat{h}_f(a)=\log(\hat{H}_f(a))$ is either $0$ or transcendental thanks to
Lindemann's theorem \cite{Lin82_UD}. Other than this, very little is known about Conjecture~\ref{conj:0 or transcendental}. It appears that we do not know even one example of a rational number $a$  
such that $\hat{h}_f(a)$ is \emph{irrational} where $\displaystyle f(z)=z^2+\frac{1}{2}$.

	As immediate consequences of our main result, we resolve Problem~\ref{prob:Silverman} and make first significant progress to Conjecture~\ref{conj:0 or transcendental} in the case of polynomial dynamics.
	From now on, we consider the case $f(z)\in \Qbar[z]$. We use $\cJ_f$ and $\cK_f$ respectively to denote the Julia set and filled Julia set of $f$. Let $a\in\Qbar$. A very artificial way to force the algebraicity of $\hat{H}_f(a)$ is to require that $\sigma(a)\in \cK_{\sigma(f)}$ for every $\sigma\in \cG:=\Gal(\Qbar/\Q)$. Let $K$ denote a number field containing $a$ and the coefficients of $f$ then the above property is equivalent to having $\sup_{n}\vert f^n(a)\vert_v <\infty$
	for every archimedean place $v$ of $K$. Algebraicity of $\hat{H}_f(a)$ follows from the fact
	that the non-archimedean contribution to $\hat{H}_f(a)$ is an algebraic number while the (logarithm of the)  archimedean contribution to
	$\hat{H}_f(a)$ vanishes; this is reminiscent to the situation over function fields in which there is no archimedean contribution at all. Moreover if $a$ is not $f$-preperiodic then we have an example in which $\hat{h}_f(a)$ is transcendental thanks to Lindemann's theorem again. As mentioned in \cite{Ngu15_AI}, this strategy to produce an example works under the assumption that the interior of $\cK_{\sigma(f)}$ 
is non-empty for every $\sigma\in \cG$. On the other hand, when there exists $\sigma\in\cG$ such that
the interior of 
$\cK_{\sigma(f)}$ is empty (equivalently $\cK_{\sigma(f)}=\cJ_{\sigma(f)}$), it is a very hard problem in general to determine whether a fractal like $\cJ_{\sigma(f)}$ contains an algebraic number that is not $\sigma(f)$-preperiodic. This very hard problem is analogous to the notorious conjecture asserting the nonexistence of algebraic irrational numbers in the Cantor set.

	 Following the terminology in Favre-Gauthier book \cite{FG22_TA}, we have:
	 \begin{definition}
	 A polynomial of degree $d \ge 2$ is called \emph{integrable} if it is linearly conjugate to $z^d$ or $\pm C_d(z)$, otherwise it is said to be \emph{non-integrable}. We let $\mathcal D_d$ denote the set of non-integrable polynomials of degree $d$.
	\end{definition}
		
	Non-integrable polynomials are called disintegrated in Medvedev-Scanlon paper \cite{MS14_IV}. 
	For a polynomial $f(z)\in\C[z]$ of degree $d\geq 2$, a B\"ottcher coordinate of $f$ is a Laurent series:
	$$\phi_f(z)= a_1z+a_0+\frac{a_{-1}}{z}+\frac{a_{-2}}{z^2}+\ldots\in z\C[[1/z]]$$
	with $a_1\neq 0$ such that $\phi_f(f(z))=\phi_f(z)^d$. A B\"ottcher coordinate exists and is unique up to multiplication by a $(d-1)$-th root of unity. More details about B\"ottcher coordinates and polynomial canonical heights will be given in the next section. The basin of infinity of $f$ is the set of $a\in\C\cup\{\infty\}$ such that $\displaystyle\lim_{n\to\infty} f^n(a)=\infty$, i.e. the set $(\C\cup\{\infty\})\setminus\cK_f$. By a \emph{functional domain of convergence} of $\phi_f$, we mean a connected neighborhood $D$ of $\infty$ inside the basin of infinity such that
	$f(D)\subseteq D$ and $\phi_f$ is convergent on $D$; then we have $\phi_f(f(a))=\phi_f(a)^d$ for every $a\in D$. At first sight, our main result does not seem to have anything to do with canonical heights:
	
	\begin{theorem}\label{thm:main}
	Let $\alpha$ be an algebraic number with the following property. There exist $d\geq 2$, $r\in\N_0$, 
	 $f_1,\ldots,f_r\in\cD_d$, algebraic numbers $a_1,\ldots,a_r$ such that $a_i$ is in the 
	 functional domain of convergence of $\phi_{f_i}$ for every $i$, and 
	integers $n_1,\ldots,n_r$ such that:
	$$\alpha=\phi_{f_1}(a_1)^{n_1}\cdots\phi_{f_r}(a_r)^{n_r}.$$
	 Then $\alpha$ is a root of unity.	 
	\end{theorem}
	
	\begin{remark}
	In this paper, the superscript $n$ means both the $n$-th iterate and the $n$-th power maps. We believe that this will not cause any confusion.
	\end{remark}
	
	The first consequence of Theorem~\ref{thm:main} resolves Problem~\ref{prob:Silverman} in 
	the case of polynomial dynamics:   
	\begin{corollary}\label{cor:algebraicity of hatH}
	Let $f\in\Qbar[z]$ be a non-integrable polynomial of degree $d\geq 2$ and let $a\in \Qbar$. The following are equivalent:
	\begin{itemize}
		\item [(i)] $\hat{H}_f(a)$ is algebraic.
		\item [(ii)] $\sigma(a)\in\cK_{\sigma(f)}$ for every $\sigma\in \cG$.
		\item [(iii)] For every archimedean place $v$ of $\Qbar$, we have $\displaystyle\sup_{n}\vert f^n(a)\vert_v<\infty$. 
	\end{itemize} 	
	\end{corollary}

	
	\begin{remark}
	Our formulation of Corollary~\ref{cor:algebraicity of hatH} does not involve a number field $K$ over which $f$ and $a$ are defined. If we pick such a $K$ then in part (ii) we can replace ``$\sigma\in\cG$''
by ``embedding $\sigma:\ K\to\C$''. Similarly in part (iii), we can replace ``archimedean place $v$ of $\Qbar$'' by ``archimedean place $v$ of $K$''. 	
	  Corollary~\ref{cor:algebraicity of hatH} means the rather surprising fact that the artificial way to force the algebraicity of $\hat{H}_f(a)$ mentioned earlier is indeed the only way! 
	In other words, algebraicity of $\hat{H}_f(a)$ happens exactly when there is no archimedean contribution. This is in stark contrast to the case when $f$ is linearly conjugate to $z^d$ or $\pm C_d(z)$ in which $\hat{H}_f(a)$ is always algebraic.
	\end{remark}
	
	The second consequence of Theorem~\ref{thm:main} proves the existence of plenty of numbers having irrational logarithmic canonical height:
	\begin{corollary}\label{cor:irrational hhat}
	Let $r,d\geq 2$, let $K$ be a number field, and let $f_1(z),\ldots,f_r(z)\in K[z]$ be non-integrable polynomials of degree $d$. Let $S=\{p_1,\ldots,p_r\}$ be a set of $r$ distinct prime numbers.
	Let $a_1,\ldots,a_r\in K$ with the following property.
	For every $1\leq i\leq r$, there exists a place $v_i$ of $K$ lying above $p_i$ such that
	$\displaystyle\lim_{n\to\infty}\vert f_i^n(a_i)\vert_{v_i}=\infty$ while $\displaystyle\sup_n\vert f_i^n(a_i)\vert_{w}<\infty$ for every place $w$ of $K$ lying above any of the $p_j$ with $j\neq i$. Then the numbers $\hat{h}_{f_1}(a_1),\ldots,\hat{h}_{f_r}(a_r)$
	are linearly independent over $\Q$. Consequently, all but at most one of them are irrational.
	\end{corollary}
	
	\begin{example}
	Consider the example $f_1(z)=\ldots=f_r(z)=f(z)=\displaystyle z^2+\frac{1}{2}$ earlier and consider $\hat{h}_{f}(1/p)$ for odd prime numbers $p$.
	Corollary~\ref{cor:irrational hhat} implies that at most one such value is rational, at most two such values belong to any given quadratic field such as $\Q(\sqrt{2})$, at most three of them belong to any given cubic field, etc. One could have similar examples for an arbitrary $f(z)\in\Qbar[z]$ and this yields an answer to a more general version of \cite[Question~6.4]{Ngu15_AI}.
	\end{example}

   We now explain in detail the method in the proof of Theorem~\ref{thm:main}. We can slightly rephrase this theorem as follows: a number of the form $\displaystyle\phi_{f_1}(a_1)^{n_1}\cdots\phi_{f_r}(a_r)^{n_r}$ is either a root of unity or transcendental.
   Generally speaking, there are two common methods to prove that a given number is transcendental. The first one is to use the Schmidt's Subspace Theorem: recent examples include the paper \cite{Ngu21_TS} resolving the transcendence counterpart of an open problem by Erd\"os-Graham and the paper \cite{BDJ20_AT}  by Bell-Diller-Jonsson giving first examples of dominant rational self-maps having transcendental dynamical degrees as well as their recent work with Krieger \cite{BDJK21_BM} on birational self-maps. Roughly speaking, in the above applications of the Subspace Theorem, the given number has a special form so that it has a strong diophantine approximation property. Unfortunately, this does not seem to be the case for numbers of the form $\displaystyle\phi_{f_1}(a_1)^{n_1}\cdots\phi_{f_r}(a_r)^{n_r}$.
	
	The other method is to use the construction of certain auxiliary polynomials vanishing at certain points. This is an old and extremely useful method in transcendental number theory and diophantine approximation. In recent years, such a construction is also capable of proving several surprising results in combinatorics \cite{CLP17_PF,EG17_OL,BN21_AA} and coined the Polynomial Method \cite{Gut16_PM}. Perhaps the most relevant idea to our current work dates back to a series of papers by Mahler in the late 1920s \cite{Mah29_AE,Mah30_AE,Mah30_UV} concerning the transcendence and algebraic independence of
	values of functions satisfying certain functional equations. The ideas in Mahler's papers have been developed further since the 1970s and the readers are referred to Nishioka's notes \cite{Nis96_MF} for a survery of results up to the mid 1990s. Unlike these modern results in which Siegel's lemma is needed in order to control the size of the coefficients of the auxiliary polynomials, our construction relies on the much simpler dimension counting arguments as in \cite{CLP17_PF,EG17_OL,BN21_AA}. In return for this simplicity, the powerful Medvedev-Scanlon classification \cite{MS14_IV} is needed in the proof of Theorem~\ref{thm:main}. 
	
	Let $m\in\N$ and let $f_1,\ldots,f_m\in\Qbar[z]$. We use $(f_1,\ldots,f_m)$ to denote the split polynomial map $(\bP^1)^m\rightarrow (\bP^1)^m$ given by
	$$(x_1,\ldots,x_m)\mapsto (f_1(x_1),\ldots,f_m(x_m)).$$
	We will use three features of the Medvedev-Scanlon classification. The first one states that if $f_1,\ldots,f_m$ are non-integrable polynomials while $g_1,\ldots,g_n$ are integrable then every preperiodic subvariety of $(\bP^1)^{m+n}$ with respect to the self-map
$(f_1,\ldots,f_m,g_1,\ldots,g_n)$  is of the form $V_1\times V_2$ where $V_1\subseteq (\bP^1)^m$ is preperiodic with respect to $(f_1,\ldots,f_m)$ and $V_2\subseteq (\bP^1)^n$ is preperiodic with respect to $(g_1,\ldots,g_n)$. The second one states that if $m\geq 2$, $V$ is a preperiodic subvariety
with respect to $(f_1,\ldots,f_m)$, and $\dim(V)<m$ 
then there exist $1\leq i\neq j\leq m$ and an $(f_i,f_j)$-preperiodic curve $C$
in $(\bP^1)^2$ such that $V\subseteq \pi_{ij}^{-1}(C)$ where $\pi_{ij}:\ (\bP^1)^n\rightarrow(\bP^1)^2$
is the projection to the $i$-th and $j$-th coordinate factors. These two features are given in \cite[Theorem~2.30]{MS14_IV}. In fact, the statement of their result is much stronger: $V$ is an irreducible component of the intersection of varieties of the form
$\pi_{ij}^{-1}(C)$. Finally, the third one states that when the projection from $C$ to each of the coordinate factors $\bP^1$ is non-constant then $C$ is the genus $0$ curve parametrized by a pair of polynomials satisfying a certain semiconjugacy functional equation \cite[Proposition~2.34]{MS14_IV}.
These three features constitute the ``coarse structure'' of preperiodic subvarieties \cite[Section~2]{MS14_IV} and some of its aspects have appeared in earlier work of Chatzidakis-Hrushoski-Peterzil \cite{CHP02_MT} and Medvedev \cite{Med07_MS}. The ``finer structure'' \cite[Sections~3--6]{MS14_IV} (also see \cite{Ngu15_SA,Pak17_PS}) boils down to the description of invariant curves of $(\bP^1)^2$ 
under self-maps of the form $(f,f)$ and involves mostly elementary yet highly technical arguments in the theory of polynomial decomposition; this is not needed in the proof of Theorem~\ref{thm:main}.

{\bf Acknowledgments.} We are grateful to Professors Jason Bell, Dragos Ghioca, Thomas Scanlon, and Joseph Silverman for several helpful comments. The author is partially supported by an NSERC Discovery Grant and a CRC Research Stipend.

\section{Notations and preliminary results for polynomial dynamics}
\subsection{Canonical heights} Define $\log^+(x)=\log\max\{\vert x\vert,1\}$ for every $x$. Throughout this subsection, let $K$ be a number field, let $f(z)\in K[z]$ with $d:=\deg(f)\geq 2$, and 
let $a\in K$. Let $M_K=M_K^\infty\cup M_K^0$ where
  	$M_K^\infty$ is the set of archimedean places and 
  	$M_K^0$ is the set of finite places of $K$. For each $v\in M_K$, we normalize the absolute value $\vert\cdot\vert_v$ on $K$
  	to be the unique extension of the usual $\vert\cdot\vert_p$ on $\Q$ where $p$ is the restriction of $v$ to $\Q$ and let $n_v=[K_v:\Q_p]$, this normalization follows \cite[Chapter~3]{Sil07_TA} and differs
  	from \cite[p.~11]{BG06_HI}. We define the Green function and canonical height functions:
  	$$g_{f,v}(a)=\lim_{n\to\infty}\frac{1}{d^n}\log^+ \vert f^n(a)\vert_v,$$
    $$\hat{h}_{f}(a)=\frac{1}{[K:\Q]}\sum_{v\in M_K} n_v g_{f,v}(a),\ \text{and}$$
    $$\hat{H}_f(a)=\exp(\hat{h}_f(a)).$$
    It is well-known that the limit in the definition of $g_{f,v}$ exists, the sum in the definition of $\hat{h}_f(a)$ is a finite sum (i.e. $g_{f,v}(a)=0$ for all but finitely many $v\in M_K$), and $\hat{h}_f(a)$ is independent of the choice of $K$, see \cite[Chapter~5]{Sil07_TA}. We have:
    \begin{lemma}\label{lem:gfv=clogp}
    Suppose $v\in M_K^{0}$ restricts to $p\in M_{\Q}^0$. Then $g_{f,v}(a)=c\log p$ with $c\in\Q$.
    \end{lemma}
    \begin{proof}
  	This is well-known, we include the proof here for the convenience of the readers. We can write
  	$\displaystyle \vert f^n(a)\vert_v=p^{c_n}$ with $c_n\in\Q$. If the $c_n$'s are bounded from above
  	then $g_{f,v}(a)=0$. Otherwise, let $n_0$ be such that $\vert f^{n_0}(a)\vert_v=p^{c_{n_0}}$ is sufficiently large. Let $\ell$ be the leading coefficient of $f$ and let $c'\in\Q$ be such that
  	$\vert\ell\vert_v=p^{c'}$, then $c_{n+1}=dc_n+c'$ for every $n\geq n_0$. Put $c''=c'/(d-1)$, then the above recurrence relation gives:
  	$$c_n=d^{n-n_0}(c_{n_0}+c'')-c''\ \text{for $n\geq n_0$.}$$
  	Therefore $c=\displaystyle\lim_{n\to\infty}\frac{c_n}{d^n}=\frac{c_{n_0}+c''}{d^{n_0}}\in\Q.$
    \end{proof}
    
\subsection{B\"ottcher coordinates}
The theory of B\"ottcher coordinates is an important tool to study polynomial dynamics. Early contributors are B\"ottcher \cite{Bot04_TP} and Ritt \cite{Rit20_OT} and we refer the readers to \cite[Chapter~9]{Mil06_DI}
for classical results over $\C$. There have been several modern treatments recently \cite{Ing13_AG,DGKNTY19_BH,FG22_TA} and we follow  \cite[Section~2.4]{FG22_TA} here.

Let $K$ be a field of characteristic $0$ and let $f(z)\in K[z]$ with $d:=\deg(f)\geq 2$. A B\"ottcher coordinate of $f$ is a Laurent series:
$$\phi_f(z)=a_1z+a_0+\frac{a_{-1}}{z}+\frac{a_{-2}}{z^2}+\ldots\in \bar{K}[[1/z]]$$
with $a_1\neq 0$ such that $\phi_f(f(z))=\phi_f(z)^d$. We have the following:
\begin{proposition}\label{prop:Bottcher coordinate exists}
Suppose $K$ contains a $(d-1)$-th root of the leading coefficient $\ell$ of $f$. Then there exists
$$\phi(z)=a_1z+a_0+\ldots\in zK[[1/z]]$$
that is a B\"ottcher coordinate of $f$. Moreover, we have:
\begin{itemize}
	\item [(a)] $a_1^{d-1}=\ell$.
	\item [(b)] $\psi(z)\in \bar{K}((1/z))$ is a B\"ottcher coordinate of $f$ if and only if $\psi(z)=\zeta \phi(z)$ where $\zeta$ is a $(d-1)$-th root of unity.
\end{itemize}
\end{proposition}
\begin{proof}
See \cite[Section~2.4]{FG22_TA}.
\end{proof}

For the rest of this subsection, we assume that $K$ is a number field containing a $(d-1)$-th root of the leading coefficient of $f$. Let $\phi(z)\in zK[[1/z]]$ be a B\"ottcher coordinate of $f$. 
\begin{proposition}\label{prop:properties of Bottcher coord}
The following hold:
\begin{itemize}
	\item [(a)] For each $v\in M_K$, there exists a positive number $B_v>0$ such that for every $z\in K_v$ with $\vert z\vert_v>B_v$, we have $\lim \vert f^n(z)\vert_v=\infty$ and $\phi(z)$ is convergent. Suppose $\vert z\vert_v>B_v$ and $\vert f(z)\vert_v>B_v$ then $\phi(f(z))=\phi(z)^d$.
	
	\item [(b)] Let $v\in M_K$ and suppose $\vert a\vert_v>B_v$, then $g_{f,v}(a)=\log \vert \phi(a)\vert_v$.
	
	\item [(c)] Let $\sigma$ be an embedding of $K$ into $\C$ and let $v\in M_K^{\infty}$ be given by
	$\vert z\vert_v=\vert\sigma(z)\vert$ for every $z\in K$. Then $\sigma(\phi)$ is a B\"ottcher coordinate of $\sigma(f)$. Moreover, if  $\vert a\vert_v>B_v$ then 
	$\sigma(\phi)$ is convergent at $\sigma(a)$ and 
	$g_{f,v}(a)=\log\vert \sigma(\phi)(\sigma(a))\vert$.
\end{itemize}
\end{proposition}
\begin{proof}
	Part (a) follows from \cite[Section~2.4]{FG22_TA}. Part (b) is well-known but it is only stated in \cite[Proposition~2.13]{FG22_TA} when $f$ has a special form, so we include the short proof here for the sake of completeness. Replacing $a$ by some $f^k(a)$ if necessary, we may assume that
	$\vert f^n(a)\vert_v>B_v$ for $n\geq 0$. Write $\phi(z)=a_1z+\ldots$, then as $n\to\infty$ we have:
	$$\vert\phi(a)\vert_v^{d^n}=\vert \phi(f^n(a))\vert_v= \vert a_1\vert_v\cdot\vert f^n(a)\vert_v+O(1).$$
Taking logarithm then dividing by $d^n$ both sides and let $n\to\infty$, we get the desired result. For part (c), we have that $\sigma(\phi)$ satisfies the functional equation in the definition of a B\"ottcher coordinate for $\sigma(f)$:
$$\sigma(\phi)(\sigma(f))=\sigma(\phi(f))=\sigma(\phi^d)=\sigma(\phi)^d,$$
hence $\sigma(\phi)$ is a B\"ottcher coordinate of $\sigma(f)$.
Finally, $\sigma$ is an isomorphism from $(K,\vert\cdot\vert_v)$ and $(\sigma(K),\vert\cdot\vert)$. Hence it can be extended to an isomorphism between $(K_v,\vert\cdot\vert_v)$ and 
$(\widehat{\sigma(K)},\vert\cdot\vert)$ where $\widehat{\sigma(K)}=\R$ or $\C$ is the completion
of $\sigma(K)$ with respect to $\vert\cdot\vert$. Then we have:
$$g_{f,v}(a)=\log \vert \phi(a)\vert_v=\log\vert \sigma(\phi)(\sigma(a))\vert.$$
\end{proof}

\subsection{The Medvedev-Scanlon classification}
\begin{definition}
A curve in $\bP^1\times\bP^1$ is called non-fibered if its projection to each of the coordinates $\bP^1$ is non-constant.
\end{definition}

We have the following properties of the coarse 
structure of preperiodic subvarieties of a split polynomial map given in \cite[Section~2]{MS14_IV}:
\begin{proposition}\label{prop:MS}
Let $f_1,\ldots,f_m$ be non-integrable polynomials of degree $d\geq 2$ and let $g_1,\ldots,g_n$ be integrable polynomials of degree $d$. We have:
\begin{itemize}
	\item [(a)] Every subvariety of $(\bP^1)^{m+n}$ that is preperiodic under $(f_1,\ldots,f_m,g_1,\ldots,g_n)$ has the form $V\times W$ where $V$ is a subvariety of $(\bP^1)^m$ that is preperiodic under
	$(f_1,\ldots,f_m)$ and $W$ is a subvariety of $(\bP^1)^n$ that is preperiodic under $(g_1,\ldots,g_n)$.
	
	\item [(b)] Let $V$ be a subvariety of $(\bP^1)^m$ that is preperiodic under $(f_1,\ldots,f_m)$. Suppose $m\geq 2$ and $\dim(V)<m$ then there exist an $(f_i,f_j)$-preperiodic curve $C$ in $(\bP^1)^2$
	such that $V\subseteq \pi_{ij}^{-1}(C)$ where $\pi_{ij}:(\bP^1)^m\rightarrow (\bP^1)^2$
	is the projection onto the $i$-th and $j$-th coordinate factors.
	
	\item [(c)] Let $f$ and $g$ be non-integrable polynomials of degree $d\geq 2$ and let $C$ be a non-fibered curve that is invariant under $(f,g)$. Then there exist non-constant polynomials $A,B,Q$ such that 
	$C=\{(A(t),B(t)):\ t\in\bP^1(\Qbar)\}$, $f\circ A=A\circ Q$, and $g\circ B=B\circ Q$.
\end{itemize}
\end{proposition}
\begin{proof}
Parts (a) and (b) follow from \cite[Theorem~2.30]{MS14_IV} while part (c) follows from \cite[Proposition~2.34]{MS14_IV}.
\end{proof}

Our next result connects the functional equation in part (c) of Proposition~\ref{prob:Silverman} to
a relationship between B\"ottcher coordinates:
\begin{proposition}\label{prop:curves to Bottcher coordinates}
Let $f(z)\in \C[z]$ be a polynomial of degree $d\geq 2$. Let $A(z),Q(z)\in\C[z]$ be non-constant such that $f\circ A=A\circ Q$. Let $\phi(z)$ and $\psi(z)$ be respectively a  B\"ottcher coordinate of $f$ and $Q$ and let $\delta=\deg(A)$. Then $\displaystyle\frac{\phi(A(z))}{\psi(z)^\delta}$ is a root of unity.
\end{proposition}
\begin{proof}
First, we observe that if $\displaystyle S(z)\in\frac{1}{z}\C[[1/z]]$ and $m$ is a positive integer then
$$(1+S(z))^{1/m}=1+\frac{1}{m}S(z)+\frac{1}{2}\cdot\frac{1}{m}\cdot\left(\frac{1}{m}-1\right) S(z)^2+\ldots$$
gives a well-defined element of $\C[[1/z]]$ and the $m$-th power of this element is $1+S(z)$. Write
$$\phi(A(z))=\alpha z^{\delta}+ T(z)$$
with $\alpha\neq 0$ and $T(z)\in z^{\delta-1}\C[[1/z]]$. Choose a $\delta$-th root $\beta$ of $\alpha$, then we have:
$$\eta(z):=\beta z\left(1+\frac{T(z)}{z^{\delta}}\right)^{1/\delta}$$
satisfies $\eta(z)^{\delta}=\phi(A(z))$. Thanks to the given functional equation, we have:
$$\eta(Q(z))^{\delta}=\phi(A(Q(z)))=\phi(f(A(z)))=\phi(A(z))^{d}=\eta(z)^{d\delta}.$$
Therefore $\eta(Q(z))=\zeta \eta(z)^d$ for some $\delta$-th root of unity $\zeta$. Put $\mu(z)=\zeta_1\eta(z)$ where $\zeta_1^{d-1}=\zeta$ then we have:
$$\mu(Q(z))=\zeta_1\eta(Q(z))=\zeta_1\zeta \eta(z)^d=\zeta_1\zeta (\mu(z)/\zeta_1)^d=\mu(z)^d.$$
Since $\deg(Q)=\deg(f)=d$, the above functional equation implies that $\mu(z)$ is a B\"ottcher coordinate of $Q(z)$. Proposition~\ref{prop:Bottcher coordinate exists} gives that $\displaystyle\frac{\mu(z)}{\psi(z)}$ is a root of unity. Raising this to the $\delta$-th power, we get the desired result.
\end{proof}
	
	\section{Proof of Theorem~\ref{thm:main}}
	We prove this theorem by contradiction. Suppose there exist $\alpha$, $d$, $r$, the $f_i$'s, $a_i$'s, and $n_i$'s as in the statement of Theorem~\ref{thm:main} with
	\begin{equation}\label{eq:pr of thm alpha=...}
	\alpha=\phi_{f_1}(a_1)^{n_1}\cdots\phi_{f_r}(a_r)^{n_r}
	\end{equation}
	where $\alpha$ is \emph{not} a root of unity and $r$ is smallest possible.
	To simplify the notation, we write $\phi_i$ instead of $\phi_{f_i}$. Let $K$ be a number field such that $\alpha$, the $f_i$'s, $a_i$'s, and $\phi_i$'s are defined over $K$ and let $\cO_K$ be its ring of integers. Obviously, $r>0$ otherwise the RHS of \eqref{eq:pr of thm alpha=...} is $1$. Moreover each $n_i\neq 0$ thanks to the minimality of $r$. We also have that $\alpha\neq 0$ since each $\phi_i(a_i)\neq 0$ thanks to the identities
	$\phi_i(f_i^k(a_i))=\phi_i(a_i)^{d^k}$, $f_i^k(a_i)\to\infty$ as $k\to\infty$, and $\phi_i(\infty)=\infty$. In the various constructions and estimates below, 
	$C$ denotes a large positive integer depending on the initial data and $L$ denotes a large positive integer depending on $C$ and the initial data.  These $C$ and $L$ are fixed and we will describe how to choose them later. After fixing $C$ and $L$, we use $k$ to denote a sufficiently large integer.

	Put $\Phi(X_1,\ldots,X_r)=\phi_1(X_1)^{n_1}\cdots\phi_r(X_r)^{n_r}$.
	Consider an auxiliary function of the form:
	$$\cA(X_1,\ldots,X_r):=P_1\Phi+P_2\Phi^2+\cdots+P_L\Phi^L$$
	where each $P_{\ell}\in \cO_K[X_1,\ldots,X_r]$
	 has degree at most $L$ in each of the variables $X_1,\ldots,X_r$ for every $1\leq \ell\leq L$. 
	
	Put
	$n:=\displaystyle\max_{1\leq i\leq r}n_i$. Then $\cA$ is a sum of terms of the form
	$X_1^{\delta_1}\cdots X_r^{\delta_r}$ where $\delta_i\in\Z$ and 
	$\delta_i \leq (n+1)L$ for $1\leq i\leq r$. The number of tuples $(\delta_1,\ldots,\delta_r)$
	satisfying the property:
	\begin{equation}\label{eq:delta_i not too small for all i}
	\delta_i\in [-CL,(n+1)L]\ \text{for every $1\leq i\leq r$}
	\end{equation}
	is at most $(2CL)^r$ as long as $C>n+1$. Having the coefficient of $X_1^{\delta_1}\cdots X_r^{\delta_r}$ vanish for every $(\delta_1,\ldots,\delta_r)$ satisfying \eqref{eq:delta_i not too small for all i} is the same as having the coefficients of the $P_i$'s satisfy a homogeneous system of at most  
	$(2CL)^r$ many linear equations defined over $K$. Since there are more than $L^{r+1}$ many such 
	coefficients, when $L>(2C)^r$ we can always find $P_1,\ldots,P_L$ not all of which are zero such 
	that the coefficient of every $X_1^{\delta_1}\cdots X_r^{\delta_r}$ in $\cA$ where the $\delta_i$'s satisfy \eqref{eq:delta_i not too small for all i} is zero. Therefore for every $(x_1,\ldots,x_r)\in\C^r$
	such that $\vert x_i\vert$ is sufficiently large for every $i$, we have:
	\begin{equation}\label{eq:upper bound for cA}
	\vert\cA(x_1,\ldots,x_r)\vert\ll \max\{1\leq i\leq r:\ \vert x_i\vert^{-CL}\prod_{j\neq i}\vert x_j\vert^{(n+1)L}\}.
	\end{equation}
	
	Let $P(X_1,\ldots,X_r,Y)=P_1Y+P_2Y^2+\ldots+P_LY^L$ which is a non-zero element of $K[X_1,\ldots,X_r,Y]$ since some $P_i$ is non-zero. Identity \eqref{eq:pr of thm alpha=...} together with the identity $\displaystyle\phi_i(f_i^k(a_i))=\phi_i(a_i)^{d^k}$ for $1\leq i\leq r$ yield:
	$$\alpha^{d^k}=\Phi(f_1^k(a_1),\ldots,f_r^k(a_r))\ \text{for every $k\in\N_0$.}$$
	Therefore
	\begin{align*}
	\begin{split}
	\vert P(f_1^k(a_1),\ldots,f_r^k(a_r),\alpha^{d^k})\vert&=\vert\cA(f_1^k(a_1),\ldots,f_r^k(a_r))\vert\\
	&\ll \max\left\{1\leq i\leq r:\ \vert f_i^k(a_i)\vert^{-CL}\prod_{j\neq i}\vert f_j^k(a_j)\vert^{(n+1)L}\right\}
	\end{split}
	\end{align*}
	when $k$ is sufficiently large so that each $\vert f_i^k(a_i)\vert$ is sufficiently large. Let $C_1,C_2>1$ be real numbers depending only on the $f_i$'s, the $a_i$'s, and $\alpha$ such that 
	\begin{equation}\label{eq:C1 and C2}
	C_1^{d^k}<\vert \sigma(f_i^k(a_i))\vert < C_2^{d^k}\ \text{and}\ \vert\sigma(\alpha)\vert <C_2
	\end{equation}
	for $1\leq i\leq r$, $\sigma\in\cG$, and for every large integer $k$. Combining this with the previous inequality, we now have:
	\begin{equation}\label{eq:very strong upper bound}
	\vert P(f_1^k(a_1),\ldots,f_r^k(a_r),\alpha^{d^k})\vert\ll C_1^{-CLd^k}C_2^{r(n+1)Ld^k}
	\end{equation}
	for every large integer $k$ where the implied constants are independent of $k$. Put:
	$$\Vert P\Vert=\max\{\vert\sigma(c)\vert:\ \sigma\in \cG\ \text{and $c$ is a coefficient of a monomial term in $P$}\}$$
	and let $C_3$ be a positive integer such that $C_3^{d^k}\alpha^{d^k}$ and the $C_3^{d^k}f_i^k(a_i)$'s
	are algebraic integers for $1\leq i\leq r$.
	
	For every $\sigma\in \cG$ and for $1\leq i\leq L$, we have
	$$\vert \sigma(P_i(f_1^k(a_1),\ldots,f_r^k(a_r)))\vert \leq \Vert P\Vert (L+1)^r C_2^{Lrd^k}$$
	thanks to \eqref{eq:C1 and C2}, the definition of $\Vert P\Vert$, and the given properties of $P_i$. Then we have:
	\begin{align}\label{eq:upper bound sigma(P)}
	\begin{split}
		\vert \sigma(P(f_1^k(a_1),\ldots,f_r^k(a_r),\alpha^{d^k}))\vert&=\left\vert\sum_{i=1}^L \sigma(P_i(f_1^k(a_1),\ldots,f_r^k(a_r)))\sigma(\alpha)^{id^k}\right\vert\\
		&\leq L\Vert P\Vert (L+1)^r C_2^{(r+1)Ld^k}.
	\end{split}
	\end{align}
	
	Put $D=[K:\Q]$ then \eqref{eq:very strong upper bound} and \eqref{eq:upper bound sigma(P)}
	yields:
	\begin{equation}\label{eq:upper bound on Norm K/Q}
	\vert \Norm_{K/\Q}(P(f_1^k(a_1),\ldots,f_r^k(a_r),\alpha^{d^k}))\vert\ll C_1^{-CLd^k}C_2^{(r(n+1)+(r+1)D)Ld^k}
	\end{equation}
	for all sufficiently large $k$; we emphasize again that the implied constants are independent of $k$.
	From the choice of $C_3$ and the given properties of the $P_i$'s, we have:
	\begin{equation}\label{eq:upper bound on denominator}
	C_3^{D(r+1)Ld^k}\cdot \Norm_{K/\Q}(P(f_1^k(a_1),\ldots,f_r^k(a_r),\alpha^{d^k}))\in\Z
	\end{equation}
	for every large $k$.
	We now choose $C$ and $L$ such that:
	\begin{align}
	\begin{split}
		&\text{the earlier inequalities}\ C>n+1\ \text{and}\ L>(2C)^r\ \text{hold and}\\
		&C_4:=C_1^{-C}C_2^{r(n+1)+(r+1)D}<\frac{1}{C_3^{D(r+1)}}.
	\end{split}
	\end{align}
	With this choice, \eqref{eq:upper bound on Norm K/Q} and \eqref{eq:upper bound on denominator} implies that there exists a positive integer $N$ such that
	\begin{equation}\label{eq:P(...)=0}
	P(f_1^k(a_1),\ldots,f_r^k(a_r),\alpha^{d^k})=0\ \text{for every integer $k\geq N$}.
	\end{equation}
	
	Let $\varphi: (\bP^1)^{r+1}\rightarrow (\bP^1)^{r+1}$ be given by
	$$\varphi(x_1,\ldots,x_r,y)=(f_1(x_1),\ldots,f_r(x_r),y^d)$$
	and let $x=(a_1,\ldots,a_r,\alpha)$. Then \eqref{eq:P(...)=0} means $\varphi^k(x)$ belongs to the proper Zariski closed set defined by $P=0$ for every $k\geq N$. 
	At this point, it is an easy exercise to show that $x$ belongs to a proper $\varphi$-preperiodic subvariety of $(\bP^1)^{r+1}$ and we include the short proof here for the convenience of the readers.
	Let $Z$ be the Zariski closure of the $\varphi^k(x)$'s with $k\geq N$. Among all the irreducible components of $Z$, let $Z'$ be one with the largest dimension. Then the set $\{\varphi^k(x):\ k\geq N\}\cap Z'$ is Zariski dense in $Z'$; otherwise we could replace $Z'$ by the Zariski closure of this set and have a smaller Zariski closed set than $Z$ containing all the $\varphi^k(x)$ for $k\geq N$. Then for every $m\geq 0$, the set $\{\varphi^{k+m}(x):\ k\geq N\}\cap \varphi^m(Z')$ is Zariski dense in $\varphi^m(Z')$ and this implies $\varphi^m(Z')\subseteq Z$. By the maximality of $\dim(Z')$, we must have that $\varphi^m(Z')$ is an irreducible component of $Z$ for every $m\geq 0$. This proves that $Z'$ is $\varphi$-preperiodic.

	Since $\alpha\neq 0$ and $\alpha$ is not a root of unity, Proposition~\ref{prop:MS}(a) implies that
	$Z'=V\times \bP^1$ where $V$ is a proper subvariety of $(\bP^1)^r$ that is preperiodic under
	$(f_1,\ldots,f_r)$. We must have $r\geq 2$ since otherwise $V$ is a preperiodic point and this is a contradiction since $a_1$ is not $f_1$-preperiodic. The projection from $V$ to each coordinate factor $\bP^1$ is non-constant since each $a_i$ is not $f_i$-preperiodic. Proposition~\ref{prop:MS}(b) now implies that there exists $1\leq i\neq j\leq r$ such that $(a_i,a_j)$ belongs to a non-fibered 
	curve in $(\bP^1)^2$ that is preperiodic under $(f_i,f_j)$.
	
	For the rest of this section, fix $k\in\N_0$ and $m\in \N$ such that $(f_i^k(a_i),f_j^k(a_j))$
	belongs to a non-fibered curve $\Gamma$ in $(\bP^1)^2$ that is invariant under $(f_i^m,f_j^m)$.
	By Proposition~\ref{prop:MS}(c) there exist non-constant polynomials $A$, $B$, and $Q$ such that $\Gamma$ is parametrized by $(A,B)$ and
	$f_i^m\circ A=A\circ Q$ and $f_j^m\circ B=B\circ Q$.
	Let $t\in \Qbar$ such that $A(t)=f_i^k(a_i)$ and $B(t)=f_j^k(a_j)$. Let $\psi$ be a B\"ottcher coordinate of $Q$, then Proposition~\ref{prop:curves to Bottcher coordinates} implies that
	$\displaystyle\frac{\phi_i(f_i^k(a_i))}{\psi(t)^{\deg(A)}}$ and 
	$\displaystyle\frac{\phi_j(f_j^k(a_j))}{\psi(t)^{\deg(B)}}$
	are roots of unity. Hence there is a relation of the form:
	$$\phi_i(a_i)^{d^k\deg (B)}=\zeta \phi_j(a_j)^{d^k\deg(A)}$$
	where $\zeta$ is a root of unity.
	We now use this relation to eliminate $\phi_i(a_i)$ from \eqref{eq:pr of thm alpha=...} to lower the value of $r$. This contradicts the minimality of $r$ and we finish the proof.
	
	\section{Proof of the corollaries and comments about further work}
	\subsection{Proof of Corollary~\ref{cor:algebraicity of hatH}} Let $K$ be a number field such that $K/\Q$ is Galois, $f\in K[z]$, $a\in K$, and $K$ contains a $(d-1)$-th root of the leading coefficient of $f$. Although this is not strictly necessary, we enlarge $K$ so that it has no real embedding. Let $\phi(z)\in zK[[1/z]]$ be a B\"ottcher coordinate of $f$.
	
	As explained before, (ii) and (iii) are equivalent to each other since archimedean places of $K$ correspond to pairs of complex conjugate embeddings. We have (iii) implies (i) since there is no archimedean contribution to $\hat{H}_f(a)$ while the non-archimedean contribution is algebraic thanks to Lemma~\ref{lem:gfv=clogp}. 
	
	It remains to prove that (i) implies (iii). We prove this by contradiction: suppose that the set
	$$S=\{v\in M_K^{\infty}:\ \sup_n \vert f^n(a)\vert_v=\infty\}$$
	is non-empty. For each $v\in S$, let $B_v$ be as in Proposition~\ref{prop:properties of Bottcher coord}. Replacing $a$ by some $f^k(a)$ if necessary, we may assume that $\vert f^n(a)\vert_v>B_v$ for every $n\in\N_0$ and every $v\in S$. List elements of $S$ as $v_1,\ldots,v_r$; note that $n_{v_i}=2$ for every $i$ since each $v_i$ is complex. For $1\leq i\leq r$, let
	$\sigma_i$ and $\bar{\sigma}_i$ be the pair of complex conjugate embeddings in $\Gal(K/\Q)$ that 
	correspond to $v_i$ meaning $\vert x\vert_{v_i}=\vert\sigma_i(x)\vert$ for every $x\in K$. 
	For $1\leq i\leq r$, we choose the B\"ottcher coordinates $\phi_{\sigma_i(f)}$ and $\phi_{\bar{\sigma}_i(f)}$
	of $\sigma_i(f)$ and $\bar{\sigma}_i(f)$ so that they are complex conjugate Laurent series. 
	Let
	$$H_0=\exp\left(\frac{1}{[K:\Q]}\sum_{v\in M_K^0} n_vg_{f,v}(a)\right)$$
	be the non-archimedean contribution to $\hat{H}_f(a)$ which is algebraic thanks to Lemma~\ref{lem:gfv=clogp}. From Proposition~\ref{prop:properties of Bottcher coord}, we have:
	$$\hat{H}_f(a)=H_0\cdot \left(\prod_{i=1}^r \phi_{\sigma_i(f)}(\sigma_i(a))\phi_{\bar{\sigma}_i(f)}(\bar{\sigma}_i(a))\right)^{1/[K:\Q]}.$$
	Theorem~\ref{thm:main} implies that $\hat{H}_f(a)/H_0$ is a root of unity but this is impossible since
	each 
	$$\vert\phi_{\sigma_i(f)}(\sigma_i(a))\vert=\vert\phi_{\bar{\sigma}_i(f)}(\bar{\sigma}_i(a))\vert=\vert\phi_f(a)\vert_{v_i}>1$$
	and we finish the proof.
	
	\subsection{Proof of Corollary~\ref{cor:irrational hhat}} We assume that there is a non-trivial linear relation $\displaystyle\sum_{i=1}^r c_i\hat{h}_{f_i}(a_i)=0$ with $c_i\in\Z$ for every $i$ and arrive at a contradiction. This yields the multiplicative relation $\displaystyle\prod_{i=1}^r\hat{H}_{f_i}(a_i)^{c_i}=1$. For each $1\leq i\leq r$, let
	$$H_{i,0}=\exp\left(\frac{1}{[K:\Q]}\sum_{v\in M_K^0} n_vg_{f_i,v}(a_i)\right)\ \text{and}\ H_{i,\infty}=\exp\left(\frac{1}{[K:\Q]}\sum_{v\in M_K^\infty} n_vg_{f_i,v}(a_i)\right)$$
	be respectively the non-archimedean and archimedean contribution to $\hat{H}_{f_i}(a_i)$.
	 The earlier multiplicative relation becomes:
	\begin{equation}\label{eq:mult relation becomes} 
	 \prod_{i=1}^r H_{i,\infty}^{c_i}=\prod_{i=1}^r H_{i,0}^{-c_i}.
	\end{equation}
	The given local condition implies that the $H_{i,0}$'s are multiplicatively independent, hence the RHS of \eqref{eq:mult relation becomes} is a positive real algebraic number that is not $1$. As in the previous subsection, we can express the LHS of \eqref{eq:mult relation becomes} into the form
	$$\prod_{j=1}^m \phi_{g_j}(b_j)^{\gamma_j}$$
	where each $(g_j,b_j)$ is Galois conjugate to some $(f_i,a_i)$ and each $\gamma_j$ is a rational number; we allow the possibility that $\sup_{n} \vert f_i^n(a_i)\vert_v<\infty$ for $1\leq i\leq r$ and $v\in M_K^\infty$ in which $m=0$ and the above expression is the  empty product. Theorem~\ref{thm:main} now implies that the RHS of \eqref{eq:mult relation becomes} is a root of unity, contradiction.
	
	\subsection{Further comments} Results in this paper are just the beginning and we end this paper with comments on some further direction. The most natural continuation of Theorem~\ref{thm:main} is to give a characterization of $(f_1,a_1),\ldots,(f_r,a_r)$ so that 
	$$\phi_{f_1}(a_1)^{n_1}\cdots\phi_{f_r}(a_r)^{n_r}=1$$
	where $n_1,\ldots,n_r\in\Z$ not all of which are $0$. Our work indicates that the answer should be when there exist $1\leq i\neq j\leq r$ such that $(a_i,a_j)$ belongs to an $(f_i,f_j)$-preperiodic curve. In the proof of Theorem~\ref{thm:main}, one has the auxiliary polynomial $P=P_1Y+\ldots+P_LY^L$ and it is obvious that $P$ is non-zero when one of the $P_i$'s is non-zero. However, if we imitate the same construction in this further problem, we will have $P=P_1+\ldots+P_r$ and it is possible that $P=0$. Therefore some further machinery or even an entirely different construction of auxiliary function is needed. Once the above problem is solved, we can give a characterization of the $(f_i,a_i)$'s 
	so that the $\hat{h}_{f_i}(a_i)$'s are linearly dependent over $\Q$ without the further local conditions as in Corollary~\ref{cor:irrational hhat}. Professor Jason Bell also suggests the possibility of strengthening existing results for the Dyanmical Mordell-Lang problem \cite{BGT16_TD} for the split polynomial map $(f,g)$ on $\bP^1\times\bP^1$.
	
	After all these, there is yet another highly interesting direction. It is usually the case that results in diophantine geometry motivate those in arithmetic dynamics, on the other hand once we establish the above results we can speculate what might happen in diophantine geometry. Consider points $A_1$ and $A_2$ on elliptic curves $E_1$ and $E_2$ respectively and let $\hat{h}_i$ be the N\'eron-Tate canonical height on $E_i$ and the question is when 
	$\hat{h}_1(A_1)$ and $\hat{h}_2(A_2)$ are linearly dependent over $\Q$. Based on what happens in the above arithmetic dynamics, one might speculate that except for the special case when either $A_1$ or $A_2$ is torsion and up to applying an automorphism $\sigma\in\cG$ the answer is when $(A_1,A_2)$ belongs to a non-fibered torsion translate of an elliptic curve in $E_1\times E_2$.
	
	\bibliographystyle{amsalpha}
	\bibliography{Transcendence} 	

\end{document}